\documentclass[11pt]{article}
\usepackage{amsmath,amsfonts,amssymb,amsthm}
\usepackage{enumerate}
\usepackage{caption}
\usepackage{pbox}

\numberwithin{equation}{section}
\theoremstyle{plain}
\newtheorem{theorem}[equation]{Theorem}
\newtheorem{corollary}[equation]{Corollary}

\newtheorem{proposition}[equation]{Proposition}

\usepackage{geometry}
\theoremstyle{definition}
\newtheorem{definition}[equation]{Definition}

\newtheorem{example}[equation]{Example}
\newtheorem{remark}[equation]{Remark}

\numberwithin{equation}{section}

\newcommand{\R}{{\mathbb R}}
\newcommand{\N}{{\mathbb N}}

\newcommand{\Om}{\Omega}

\providecommand{\vint}[1]{\mathchoice
          {\mathop{\vrule width 5pt height 3 pt depth -2.5pt
                  \kern -9pt \kern 1pt\intop}\nolimits_{\kern -5pt{#1}}}
          {\mathop{\vrule width 5pt height 3 pt depth -2.6pt
                  \kern -6pt \intop}\nolimits_{\kern -3pt{#1}}}
          {\mathop{\vrule width 5pt height 3 pt depth -2.6pt
                  \kern -6pt \intop}\nolimits_{\kern -3pt{#1}}}
          {\mathop{\vrule width 5pt height 3 pt depth -2.6pt
                  \kern -6pt \intop}\nolimits_{\kern -3pt{#1}}}}

\newcommand{\loc}{\mathrm{loc}}

\newcommand{\BV}{\mathrm{BV}}

\newcommand{\liploc}{\mathrm{Lip}_{\mathrm{loc}}}

\newcommand{\ch}{\text{\raise 1.3pt \hbox{$\chi$}\kern-0.2pt}}

\DeclareMathOperator{\capa}{Cap}
\DeclareMathOperator{\rcapa}{cap}

\DeclareMathOperator{\dist}{dist}

\DeclareMathOperator{\Lip}{Lip}

\DeclareMathOperator{\fint}{fine-int}

\begin{document}
\title{Federer's characterization of  \\
sets of finite
perimeter in metric spaces
\footnote{
\hfill \break {\it Keywords\,}: metric measure space,
set of finite perimeter, Federer's characterization, measure-theoretic boundary,
codimension one Hausdorff measure, fine topology.
}}
\author{Panu Lahti}
\maketitle

\begin{abstract}
Federer's characterization of sets of finite perimeter states
(in Euclidean spaces)
that a set is of finite perimeter if and only if the measure-theoretic boundary of the set has finite Hausdorff measure
of codimension one.
In complete metric spaces
that are equipped with a doubling measure
and support a Poincar\'e inequality, the ``only if''
direction was shown by Ambrosio (2002).
By applying fine potential theory in the case $p=1$,
we prove that the ``if'' direction holds as well.
\end{abstract}

\section{Introduction}

In the past two decades, there has been great interest in studying problems
of first-order analysis in the setting of general metric measure spaces,
see e.g. \cite{A1,AMP,BB,HK,M,S}. In particular, Sobolev functions
(sometimes called Newton-Sobolev functions in the metric setting)
and functions of bounded variation ($\BV$ functions) have been topics
of central interest. In much of the literature (as well as in the current paper) one assumes that the space is complete,
equipped with a doubling measure, and supports a Poincar\'e inequality;
see Section \ref{sec:preliminaries} for definitions.
Studying questions in such an abstract setting provides an opportunity to unify the
theories developed in specific settings such as weighted Euclidean spaces,
Riemannian manifolds, Carnot groups, etc.
Moreover, without having the Euclidean structure available, one is forced
to develop novel methods and proofs, giving new insight into various problems.

In the theory of $\BV$ functions in the Euclidean setting,
a key result originally due to De Giorgi states that
if $E$ is a set of finite perimeter, then the perimeter measure $P(E,\cdot)$
coincides with the $n-1$-dimensional Hausdorff measure restricted to the
measure-theoretic boundary $\partial^*E$. In particular, $P(E,\R^n)<\infty$
implies $\mathcal H^{n-1}(\partial^*E)<\infty$.
By a deep result due to Federer \cite[Section 4.5.11]{Fed},
the converse holds as well,
so in fact $P(E,\R^n)<\infty$ if and only if
$\mathcal H^{n-1}(\partial^*E)<\infty$.
This is known as Federer's characterization of sets of finite perimeter.
In the metric setting, where it is natural to formulate this kind of result
by means of the \emph{codimension one} Hausdorff measure $\mathcal H$,
 the ``only if'' direction of the characterization was
 shown by Ambrosio \cite{A1}, but the ``if'' direction has remained 
 an open problem.
In the current paper, we show that this direction holds as well.

\begin{theorem}\label{thm:Federers characterization}
Let $\Om\subset X$ be an open set, let $E\subset X$ be a $\mu$-measurable set, and
suppose that $\mathcal H(\partial^*E\cap \Om)<\infty$. Then $P(E,\Om)<\infty$.
\end{theorem}

The ``only if'' direction of Federer's characterization
is part of a more general structure theorem for sets of finite
perimeter, which in the metric setting
states that the perimeter measure is comparable to the Hausdorff measure of codimension one restricted to the measure-theoretic boundary.
This structure theorem is an indispensable tool in analysis of sets of finite perimeter, and hence more general
$\BV$ functions as well. While not equally essential,
the ``if'' direction of Federer's characterization has a number
of applications as well. For example, in \cite{KKST} the authors proved a characterization of Newton-Sobolev functions with zero boundary values by means
of a natural Lebesgue point-type condition on the boundary. However, the proof
relied on assuming that Federer's characterization holds;
now we know that this is the case under the usual assumptions on
the space. We will discuss other
applications in Section \ref{sec:consequences}.

Previously there have been some partial results toward a proof of the ``if'' direction.
The paper \cite{KoLaSh} showed that if the metric space is assumed to
contain a ``thick''  bundle of curves between each pair of points, then
the ``if'' direction can be proved by mimicking the Euclidean proof.
In the current paper we take a completely different
approach, which relies on \emph{fine potential theory}.
In the case $1<p<\infty$, fine potential theory deals with superharmonic functions
as understood by means of the \emph{fine topology};
see the monographs  \cite{AH,HKM,MZ} for the theory and its history in
the Euclidean setting,
and the recent papers \cite{BB-OD,BBL-SS,BBL-CCK,BBL-WC} for similar results in the
metric setting.
In \cite{L-Fed},
the author proved some analogous results in the case $p=1$, by relying
on certain continuity properties of $\BV$ functions proved earlier in \cite{L-FC,LaSh}.
An application of these results led to following characterization of sets of finite perimeter,
which is in the same vein as Federer's characterization.
Below, $\partial^1 I_E$ denotes the fine boundary of $E$, or more precisely of its measure-theoretic interior; one always has $\partial^*E\subset \partial^1 I_E$.

\begin{theorem}[{\cite[Theorem 1.1]{L-Fed}}]\label{thm:Fed style characterization}
	For an open set $\Omega\subset X$ and a  $\mu$-measurable set $E\subset X$, we have
	$P(E,\Omega)<\infty$ if and only if $\mathcal H(\partial^1 I_E\cap \Omega)<\infty$.
	Furthermore, then
	$\mathcal H((\partial^1 I_E\setminus \partial^*E)\cap\Omega)=0$.
\end{theorem}

In the current paper, our main goal is to show that if $\mathcal H(\partial^*E\cap \Omega)<\infty$, then $\mathcal H((\partial^1 I_E\setminus \partial^*E)\cap\Omega)=0$ and thus
Theorem \ref{thm:Federers characterization} follows from
Theorem \ref{thm:Fed style characterization}.
The proofs will be given in Section \ref{sec:proofs}, and they rely mostly
on properties of the $1$-fine topology proved in \cite{L-Fed,L-WC}, as well as boxing inequality-type arguments.
Our methods and the underlying theory should be of interest already in Euclidean spaces,
where Federer's original argument has
remained (as far as we know) essentially the only known
proof for the characterization.

\paragraph{Acknowledgments.}
The author wishes to thank Nageswari Shan\-muga\-lingam
and Juha Kinnunen for reading the manuscript and giving
comments that helped improve the paper.

\section{Preliminaries}\label{sec:preliminaries}

In this section we introduce the standard definitions, notation,
and assumptions used in the paper.

Throughout this paper, $(X,d,\mu)$ is a complete metric space that is equip\-ped
with a metric $d$ and a Borel regular outer measure $\mu$ satisfying
a doubling property, meaning that
there exists a constant $C_d\ge 1$ such that
\[
0<\mu(B(x,2r))\le C_d\mu(B(x,r))<\infty
\]
for every ball $B(x,r):=\{y\in X:\,d(y,x)<r\}$.
We assume that $X$ consists of at least $2$ points.
Given a ball $B=B(x,r)$ and $\beta>0$, we sometimes abbreviate $\beta B:=B(x,\beta r)$.
Note that in metric spaces, a ball (as a set) does not necessarily
have a unique center point and radius, but we understand these to
be prescribed for all balls that we consider.
When we want to state that a constant $C$
depends on the parameters $a,b, \ldots$, we write $C=C(a,b,\ldots)$.

All functions defined on $X$ or its subsets will take values in $[-\infty,\infty]$.
A complete metric space equipped with a doubling measure is proper,
that is, closed and bounded sets are compact.
For any open set $\Omega\subset X$, we define $\liploc(\Om)$ as the set of
functions that are in the class $\Lip(\Om')$ for every open $\Om'\Subset\Om$;
here $\Omega'\Subset\Omega$ means that $\overline{\Omega'}$ is a
compact subset of $\Omega$.
Other local function spaces are defined analogously.

For any set $A\subset X$ and $0<R<\infty$, the restricted spherical Hausdorff content of codimension one is defined as
\[
\mathcal{H}_{R}(A):=\inf\left\{ \sum_{i=1}^{\infty}
\frac{\mu(B(x_{i},r_{i}))}{r_{i}}:\,A\subset\bigcup_{i=1}^{\infty}B(x_{i},r_{i}),\,r_{i}\le R\right\}.
\]
The codimension one Hausdorff measure of $A\subset X$ is then defined as
\[
\mathcal{H}(A):=\lim_{R\rightarrow 0}\mathcal{H}_{R}(A).
\]

By a curve we mean a nonconstant rectifiable continuous mapping from a compact interval of the real line
into $X$.
A nonnegative Borel function $g$ on $X$ is an upper gradient 
of a function $u$
on $X$ if for all curves $\gamma$, we have
\begin{equation}\label{eq:definition of upper gradient}
|u(x)-u(y)|\le \int_\gamma g\,ds,
\end{equation}
where $x$ and $y$ are the end points of $\gamma$
and the curve integral is defined by using an arc-length parametrization,
see \cite[Section 2]{HK} where upper gradients were originally introduced.
We interpret $|u(x)-u(y)|=\infty$ whenever  
at least one of $|u(x)|$, $|u(y)|$ is infinite.

Let $1\le p<\infty$ (we will work almost exclusively with $p=1$).
We say that a family of curves $\Gamma$ is of zero $p$-modulus if there is a 
nonnegative Borel function $\rho\in L^p(X)$ such that 
for all curves $\gamma\in\Gamma$, the curve integral $\int_\gamma \rho\,ds$ is infinite.
A property is said to hold for $p$-almost every curve
if it fails only for a curve family with zero $p$-modulus. 
If $g$ is a nonnegative $\mu$-measurable function on $X$
and (\ref{eq:definition of upper gradient}) holds for $p$-almost every curve,
we say that $g$ is a $p$-weak upper gradient of $u$. 
By only considering curves $\gamma$ in a set $A\subset X$,
we can talk about a function $g$ being a ($p$-weak) upper gradient of $u$ in $A$.

Given an open set $\Om\subset X$, we let
\[
\Vert u\Vert_{N^{1,p}(\Om)}:=\Vert u\Vert_{L^p(\Om)}+\inf \Vert g\Vert_{L^p(\Om)},
\]
where the infimum is taken over all $p$-weak upper gradients $g$ of $u$ in $\Om$.
The substitute for the Sobolev space $W^{1,p}$ in the metric setting is the Newton-Sobolev space
\[
N^{1,p}(\Om):=\{u:\|u\|_{N^{1,p}(\Om)}<\infty\},
\]
which was introduced in \cite{S}.
We understand a Newton-Sobolev function to be defined at every $x\in \Om$
(even though $\Vert \cdot\Vert_{N^{1,p}(\Om)}$ is then only a seminorm).
It is known that for any $u\in N_{\loc}^{1,p}(\Om)$ there exists a minimal $p$-weak
upper gradient of $u$ in $\Om$, always denoted by $g_{u}$, satisfying $g_{u}\le g$ 
almost everywhere in $\Om$, for any $p$-weak upper gradient $g\in L_{\loc}^{p}(\Om)$
of $u$ in $\Om$, see \cite[Theorem 2.25]{BB}.

The space of Newton-Sobolev functions with zero boundary values is defined as
\[
N_0^{1,p}(\Om):=\{u|_{\Om}:\,u\in N^{1,p}(X)\textrm{ and }u=0\textrm { on }X\setminus \Om\}.
\]
This class can be understood to be a subclass of $N^{1,p}(X)$ in a natural way.

The $p$-capacity of a set $A\subset X$ is defined as
\[
\capa_p(A):=\inf \Vert u\Vert_{N^{1,p}(X)},
\]
where the infimum is taken over all functions $u\in N^{1,p}(X)$ such that $u\ge 1$ in $A$.

The variational $1$-capacity of a set $A\subset \Om$
with respect to an open set $\Om\subset X$ is defined as
\[
\rcapa_1(A,\Om):=\inf \int_X g_u \,d\mu,
\]
where the infimum is taken over functions $u\in N_0^{1,1}(\Om)$ such that
$u\ge 1$ on $A$, and $g_u$ is the minimal $1$-weak upper gradient of $u$ (in $X$).
For basic properties satisfied by capacities, such as monotonicity and countable subadditivity, see \cite{BB}.

We will assume throughout the paper that $X$ supports a $(1,1)$-Poincar\'e inequality,
meaning that there exist constants $C_P>0$ and $\lambda \ge 1$ such that for every
ball $B(x,r)$, every $u\in L^1_{\loc}(X)$,
and every upper gradient $g$ of $u$,
we have
\[
\vint{B(x,r)}|u-u_{B(x,r)}|\, d\mu 
\le C_P r\vint{B(x,\lambda r)}g\,d\mu,
\]
where 
\[
u_{B(x,r)}:=\vint{B(x,r)}u\,d\mu :=\frac 1{\mu(B(x,r))}\int_{B(x,r)}u\,d\mu.
\]

Next we recall the definition and basic properties of functions
of bounded variation on metric spaces, following \cite{M}. See also e.g.
\cite{AFP, EvGa, Fed, Giu84, Zie89} for the classical 
theory in the Euclidean setting.
Let $\Om\subset X$ be an open set. Given a function $u\in L^1_{\loc}(\Om)$,
we define the total variation of $u$ in $\Om$ as
\[
\|Du\|(\Om):=\inf\left\{\liminf_{i\to\infty}\int_\Om g_{u_i}\,d\mu:\, u_i\in N^{1,1}_{\loc}(\Om),\, u_i\to u\textrm{ in } L^1_{\loc}(\Om)\right\},
\]
where each $g_{u_i}$ is the minimal $1$-weak upper gradient of $u_i$
in $\Om$.
(In \cite{M}, local Lipschitz constants were used instead of upper gradients, but
the properties of the total variation can be proved similarly with either definition.)
We say that a function $u\in L^1(\Om)$ is of bounded variation, 
and denote $u\in\BV(\Om)$, if $\|Du\|(\Om)<\infty$.
For an arbitrary set $A\subset X$, we define
\[
\|Du\|(A):=\inf\{\|Du\|(W):\, A\subset W,\,W\subset X
\text{ is open}\}.
\]
If $u\in L^1_{\loc}(\Om)$ and $\Vert Du\Vert(\Omega)<\infty$, $\|Du\|(\cdot)$ is
a Radon measure on $\Omega$ by \cite[Theorem 3.4]{M}.
A $\mu$-measurable set $E\subset X$ is said to be of finite perimeter if $\|D\ch_E\|(X)<\infty$, where $\ch_E$ is the characteristic function of $E$.
The perimeter of $E$ in $\Omega$ is also denoted by
\[
P(E,\Omega):=\|D\ch_E\|(\Omega).
\]

Applying the Poincar\'e inequality to sequences of approximating locally
Lipschitz functions in the definition of the total variation, we get
the following $\BV$ version:
for every ball $B(x,r)$ and every 
$u\in L^1_{\loc}(X)$, we have
\[
\vint{B(x,r)}|u-u_{B(x,r)}|\,d\mu
\le C_P r\, \frac{\Vert Du\Vert (B(x,\lambda r))}{\mu(B(x,\lambda r))}.
\]
For a $\mu$-measurable set $E\subset X$, this implies
the relative isoperimetric inequality
\begin{equation}\label{eq:relative isoperimetric inequality}
\min\{\mu(B(x,r)\cap E),\,\mu(B(x,r)\setminus E)\}\le 2 C_P rP(E,B(x,\lambda r));
\end{equation}
see e.g. \cite[Equation (3.1)]{KoLa}.

The measure-theoretic interior of a set $E\subset X$ is defined as
\[
I_E:=
\left\{x\in X:\,\lim_{r\to 0}\frac{\mu(B(x,r)\setminus E)}{\mu(B(x,r))}=0\right\},
\]
and the measure-theoretic exterior as
\[
O_E:=
\left\{x\in X:\,\lim_{r\to 0}\frac{\mu(B(x,r)\cap E)}{\mu(B(x,r))}=0\right\}.
\]
The measure-theoretic boundary $\partial^{*}E$ is defined as the set of points
$x\in X$
at which both $E$ and its complement have strictly positive upper density, i.e.
\begin{equation}\label{eq:measure theoretic boundary}
\limsup_{r\to 0}\frac{\mu(B(x,r)\cap E)}{\mu(B(x,r))}>0\quad
\textrm{and}\quad\limsup_{r\to 0}\frac{\mu(B(x,r)\setminus E)}{\mu(B(x,r))}>0.
\end{equation}
Then $X=I_E\cup O_E\cup \partial^*E$.

For an open set $\Omega\subset X$ and a $\mu$-measurable set $E\subset X$ with $P(E,\Omega)<\infty$, we know that for any Borel set $A\subset\Omega$,
\begin{equation}\label{eq:def of theta}
P(E,A)=\int_{\partial^{*}E\cap A}\theta_E\,d\mathcal H,
\end{equation}
where
$\theta_E\colon \partial^*E\to [\alpha,C_d]$ with $\alpha=\alpha(C_d,C_P,\lambda)>0$, see \cite[Theorem 5.3]{A1} 
and \cite[Theorem 4.6]{AMP}.

If $\Om\subset X$ is an open set and $u,v\in L^1_{\loc}(\Om)$, then
\begin{equation}\label{eq:variation of min and max}
\Vert D\min\{u,v\}\Vert(\Om)+\Vert D\max\{u,v\}\Vert(\Om)\le
\Vert Du\Vert(\Om)+\Vert Dv\Vert(\Om);
\end{equation}
for a proof see e.g. \cite[Lemma 3.1]{L-ZB}.

The lower and upper approximate limits of a function $u$ on $X$ are defined respectively by
\[
u^{\wedge}(x):
=\sup\left\{t\in\R:\,\lim_{r\to 0}\frac{\mu(B(x,r)\cap\{u<t\})}{\mu(B(x,r))}=0\right\}
\]
and
\[
u^{\vee}(x):
=\inf\left\{t\in\R:\,\lim_{r\to 0}\frac{\mu(B(x,r)\cap\{u>t\})}{\mu(B(x,r))}=0\right\}.
\]

Unlike Newton-Sobolev functions, we understand $\BV$ functions to be
$\mu$-equivalence classes. To consider fine properties, we need to
consider the pointwise representatives $u^{\wedge}$ and $u^{\vee}$.
We note that for $u=\ch_E$ with $E\subset X$, we have $x\in I_E$ if and only if $u^{\wedge}(x)=u^{\vee}(x)=1$, $x\in O_E$ if and only if $u^{\wedge}(x)=u^{\vee}(x)=0$, and $x\in \partial^*E$ if and only if $u^{\wedge}(x)=0$ and $u^{\vee}(x)=1$.\\

\emph{Throughout this paper we assume that $(X,d,\mu)$ is a complete metric space
	that is equipped with the doubling measure $\mu$ and supports a
	$(1,1)$-Poincar\'e inequality.}

\section{The $1$-fine topology}\label{sec:fine topology}

In this section we have gathered all the
results concerning the $1$-fine topology that
our argument will rely on. For these, we refer to \cite{L-Fed,L-FC,L-WC}.
Most of the results are analogous to those that hold in the case $1<p<\infty$,
which has been studied in the metric setting in \cite{BB-OD,BBL-CCK,BBL-WC}.

\begin{definition}\label{def:1 fine topology}
We say that $A\subset X$ is $1$-thin at the point $x\in X$ if
\[
\lim_{r\to 0}r\frac{\rcapa_1(A\cap B(x,r),B(x,2r))}{\mu(B(x,r))}=0.
\]
We also say that a set $U\subset X$ is $1$-finely open if $X\setminus U$ is $1$-thin at every $x\in U$. Then we define the $1$-fine topology as the collection of $1$-finely open sets on $X$.

We denote the $1$-fine interior of a set $H\subset X$, i.e. the largest $1$-finely open set contained in $H$, by $\fint H$. We denote the $1$-fine closure of a set $H\subset X$, i.e. the smallest $1$-finely closed set containing $H$, by $\overline{H}^1$. The $1$-fine boundary of $H$
is $\partial^1 H:=\overline{H}^1\setminus \fint H$.
Finally, the $1$-base $b_1 H$ is defined as the set of points
where $H$ is \emph{not} $1$-thin.
\end{definition}

See \cite[Section 4]{L-FC} for discussion on this definition, and for a proof of the fact that the
$1$-fine topology is indeed a topology.
By \cite[Proposition 6.16]{BB}, a set $A\subset X$ is $1$-thin at $x\in X$
if and only if
\[
\lim_{r\to 0}\frac{\rcapa_1(A\cap B(x,r),B(x,2r))}
{\rcapa_1(B(x,r),B(x,2r))}=0,
\]
and so it is clear that $W\subset b_1 W$ for any open set $W\subset X$.

Now we collect some facts concerning the $1$-fine topology proved in \cite{L-Fed}.
According to \cite[Corollary 3.5]{L-Fed}, the $1$-fine closure
of $A\subset X$ can be characterized in the following way:
\begin{equation}\label{eq:characterization of fine closure}
\overline{A}^1=A\cup b_1 A.
\end{equation}
From this it easily follows that for any $A\subset X$ and any ball $B(x,r)$, we have
$\overline{A}^1\cap B(x,r)\subset \overline{A\cap B(x,r)}^1$, and then
by \cite[Proposition 3.3]{L-Fed} we get
\begin{equation}\label{eq:variational capacity of fine closure}
\rcapa_1(\overline{A}^1\cap B(x,r),B(x,2r))
=\rcapa_1(A\cap B(x,r),B(x,2r)).
\end{equation}
By \cite[Lemma 4.6]{L-Fed} the $1$-fine boundary of a measure-theoretic
interior can be characterized as follows:
for any $\mu$-measurable set $E\subset X$,
\begin{equation}\label{eq:characterization of fine boundary}
	\partial^1 I_E=b_1 I_E\cap b_1 (X\setminus I_E).
\end{equation}
By \cite[Lemma 3.1]{L-Fed} we know that for any $\mu$-measurable set $E\subset X$,
\begin{equation}\label{eq:meas theor bdry part of fine boundary}
\partial^*E\subset \partial^1 I_E.
\end{equation}
Conversely, if $\Omega\subset X$ is open and $E\subset X$ is  $\mu$-measurable such that $P(E,\Omega)<\infty$, then
by Theorem \ref{thm:Fed style characterization},
\[
\mathcal H((\partial^1 I_E\setminus \partial^*E)\cap\Omega)=0.
\]
Combining this with \eqref{eq:def of theta} gives
\begin{equation}\label{eq:perimeter and measure of 1boundary are comparable}
\alpha \mathcal H(\partial^1 I_E\cap \Om)\le P(E,\Om)
\le C_d\mathcal H(\partial^1 I_E\cap \Om).
\end{equation}
In fact this holds for every $\mu$-measurable $E\subset X$; to see this
we can assume that $\mathcal H(\partial^1 I_E\cap \Om)<\infty$,
and then $P(E,\Om)<\infty$ by Theorem \ref{thm:Fed style characterization}.

We also have the following version of the relative isoperimetric inequality:
for every ball $B(x,r)$ and every $\mu$-measurable $E\subset X$,
\begin{equation}\label{eq:relative isoperimetric inequality for fine boundary}
\min\{\mu(B(x,r)\cap E),\,\mu(B(x,r)\setminus E)\}\le 2 C_P C_d r
\mathcal H(\partial^1 I_E\cap B(x,\lambda r));
\end{equation}
this follows from the ordinary relative isoperimetric inequality
\eqref{eq:relative isoperimetric inequality} and
\eqref{eq:perimeter and measure of 1boundary are comparable}.

\begin{remark}\label{rmk:boundaries}
It may seem strange to talk about $\partial^1 I_E$, as it seems that we are first taking
the interior in one topology and then the boundary in another. However, if we define
the measure topology more axiomatically, then $I_E$ is actually \emph{not} the interior of $E$
in the measure topology, and should be seen as a measure-theoretic quantity rather than
a topological one (see \cite[Remark 4.9]{L-Fed}). Moreover, $\partial^*E$ is actually
the boundary of $I_E$ in the measure topology; let us denote it
by $\partial^0 I_E$. Thus $\partial^1 I_E$ is a natural set to consider as well.
Finally, we can note that $\partial^1 I_E=\partial^1 O_E$, see
\cite[Lemma 4.8]{L-Fed}.
\end{remark}

The following \emph{weak Cartan property} in the case $p=1$ was proved in \cite{L-WC}.
Note that here we have not defined the concept of $1$-superminimizers, but we will
not need it at all in this paper.

\begin{theorem}[{\cite[Theorem 5.2]{L-WC}}]\label{thm:weak Cartan property in text}
	Let $A\subset X$ and let $x\in X\setminus A$ be such that $A$
	is $1$-thin at $x$.
	Then there exist $R>0$ and $E_0,E_1\subset X$ such that $\ch_{E_0},\ch_{E_1}\in\BV(X)$,
	$\ch_{E_0}$ and $\ch_{E_1}$ are $1$-superminimizers in $B(x,R)$,
	$\max\{\ch_{E_0}^{\wedge},\ch_{E_1}^{\wedge}\}=1$ in $A\cap B(x,R)$,
	$\ch_{E_0}^{\vee}(x)=0=\ch_{E_1}^{\vee}(x)$,
	$\{\max\{\ch_{E_0}^{\vee},\ch_{E_1}^{\vee}\}>0\}$ is $1$-thin at $x$,
	and
	\begin{equation}\label{eq:weak Cartan property energy thinness}
	\lim_{r\to 0}r\frac{P(E_0,B(x,r))}{\mu(B(x,r))}=0,\qquad
	\lim_{r\to 0}r\frac{P(E_1,B(x,r))}{\mu(B(x,r))}=0.
	\end{equation}
\end{theorem}

The following simpler formulation will be sufficient for our purposes.

\begin{corollary}\label{cor:weak Cartan property}
	Let $A\subset X$ and let $x\in X\setminus A$ be such that $A$
	is $1$-thin at $x$.
	Then there exist $R>0$ and $F\subset X$ such that $\ch_{F}\in\BV(X)$,
	$A\cap B(x,R)\subset I_F$,
	$I_F$ is $1$-thin at $x$, and
	\begin{equation}\label{eq:weak Cartan property energy thinness for F}
	\lim_{r\to 0}r\frac{P(F,B(x,r))}{\mu(B(x,r))}=0.
	\end{equation}
\end{corollary}
\begin{proof}
Take $E_0,E_1\subset X$ as given by
Theorem \ref{thm:weak Cartan property in text}, and set
$F:=E_0\cup E_1$. By \eqref{eq:variation of min and max} we obtain
$\ch_{F}\in\BV(X)$, and \eqref{eq:variation of min and max}
and \eqref{eq:weak Cartan property energy thinness} together give
\eqref{eq:weak Cartan property energy thinness for F}.
From the fact that $\max\{\ch_{E_0}^{\wedge},\ch_{E_1}^{\wedge}\}=1$ in $A\cap B(x,R)$
we obtain that
\[
A\cap B(x,R)\subset I_{E_0}\cup I_{E_1}\subset I_F.
\]
Finally, since $\{\max\{\ch_{E_0}^{\vee},\ch_{E_1}^{\vee}\}>0\}$ is $1$-thin at $x$, then so is
\[
I_{E_0}\cup I_{E_1}\cup \partial^*E_0\cup\partial^*E_1\supset I_F.
\]
\end{proof}

In \cite[Lemma 4.4]{L-WC} it was also shown that if $A\subset X$
is $1$-thin at a point $x\in X\setminus A$,
then there exists an open set that
contains
$A$ and is also $1$-thin at $x$, that is,
\begin{equation}\label{eq:existence of thin open set}
\textrm{If }x\notin A\cup b_1 A, \textrm{then there exists an open }W\supset A\textrm{ such that }x\notin b_1 W.
\end{equation}

\section{Proof of the characterization}\label{sec:proofs}

In \cite[Theorem 3.11]{KoLa} it was shown that
for any $\mu$-measurable set $E\subset X$,
we have $\overline{\partial^* E}=\partial I_E$,
that is, the closure of the measure-theoretic boundary (in the metric topology)
is the whole topological boundary of a suitable representative of $E$
(namely the measure-theoretic interior $I_E$).
Now we prove the analogous result with the metric topology replaced by the 
$1$-fine topology.
This will be the crux of our proof of Federer's characterization.

\begin{theorem}\label{thm:fine closure gives fine boundary}
	For any $\mu$-measurable set $E\subset X$, we have $\overline{\partial^*E}^1=\partial^1 I_E$.
\end{theorem}

Note that by Remark \ref{rmk:boundaries}, the above can be 
written as $\overline{\partial^0 I_E}^1=\partial^1 I_E$, showing that the result
describes the interplay between the measure topology and the $1$-fine topology.
It is natural to ask which other sets and topologies would satisfy an analogous property, but we will not pursue this problem here.
Previously, properties of the measure topology and fine topologies
have been studied in the monograph \cite{LMZ}.

\begin{proof}
By \eqref{eq:meas theor bdry part of fine boundary} we have
\[
\overline{\partial^*E}^1\subset \overline{\partial^1 I_E}^1=\partial^1 I_E,
\]
where the last equality follows from the fact that boundaries
are closed sets in every topology.
Thus we only need to show that $\overline{\partial^*E}^1\supset\partial^1 I_E$.
Let $x_0\in \partial^1 I_E$ and let $U\ni x_0$ be a $1$-finely open set. We need to show that
$\partial^*E\cap U\neq \emptyset$.
By \eqref{eq:existence of thin open set} there exists an open set $W\supset X\setminus U$ that is $1$-thin at $x_0$.
Since $\overline{W}^1=W\cup b_1 W=b_1 W$
by \eqref{eq:characterization of fine closure}, we have
$x_0\notin \overline{W}^1$.
We will show that $\partial^*E\setminus W\neq\emptyset$; suppose that instead 
$\partial^*E\setminus W=\emptyset$.
\paragraph{Claim.}
Let $x\in \partial^1 I_E\setminus \overline{W}^1$ and $s_1>0$. Then
there exists $y\in B(x,s_1)\cap\partial^1 I_E\setminus \overline{W}^1$ and $0<s_2\le s_1/2$ such that
\[
\frac{1}{2C_d^{\lceil\log_2(3\lambda)\rceil+1}}\le \frac{\mu(E\cap B(y,s_2))}{\mu(B(y,s_2))}\le 1-\frac{1}{2C_d^{\lceil\log_2(3\lambda)\rceil}},
\]
where $\lceil a\rceil$ is the smallest integer at least $a\in\R$.
\paragraph{Proof of claim:}
\subparagraph{Step 1.}
We can assume that $x\in O_E$; the case $x\in I_E$ is handled analogously
(recall that $\partial^1 I_E=\partial^1 O_E$ from Remark \ref{rmk:boundaries}).
Since $x\in\partial^1 I_E$, by \eqref{eq:characterization of fine boundary}
we have
\begin{equation}\label{eq:thickness of IE}
\limsup_{r\to 0}r\frac{\rcapa_1(I_E\cap B(x,r),B(x,2r))}{\mu(B(x,r))}>0.
\end{equation}
Since $x$ belongs to the $1$-finely open set $X\setminus \overline{W}^1$, we have
\[
\lim_{r\to 0}r\frac{\rcapa_1(\overline{W}^1 \cap B(x,r),B(x,2r))}{\mu(B(x,r))}=0.
\]
We apply Corollary \ref{cor:weak Cartan property} to find $R>0$ and
$F\subset X$ such that $I_F\supset \overline{W}^1\cap B(x,R)$ and
$I_F$ is $1$-thin at $x$.
Then by \eqref{eq:variational capacity of fine closure}, also
\[
\lim_{r\to 0}r\frac{\rcapa_1(\overline{I_F}^1 \cap B(x,r),B(x,2r))}{\mu(B(x,r))}=0.
\]
Combining this with \eqref{eq:thickness of IE},
by subadditivity of the variational $1$-capacity we get
\begin{equation}\label{eq:density of IE minus IF}
\limsup_{r\to 0}r\frac{\rcapa_1((I_E\setminus
	\overline{I_F}^1)\cap B(x,r),B(x,2r))}{\mu(B(x,r))}>0.
\end{equation}
According to Corollary \ref{cor:weak Cartan property}, the set $F$ also satisfies
\[
\lim_{r\to 0}r\frac{P(F,B(x,r))}{\mu(B(x,r))}=0.
\]
Thus by \eqref{eq:perimeter and measure of 1boundary are comparable}
and the doubling property of $\mu$,
\begin{equation}\label{eq:density of partial IF}
\lim_{r\to 0}r\frac{\mathcal H(\partial^1 I_F\cap B(x,2r))}{\mu(B(x,r))}=0.
\end{equation}
Combining the fact that $x\in O_E$ with \eqref{eq:density of IE minus IF}
and \eqref{eq:density of partial IF}, we find a number $a>0$ and a radius
\begin{equation}\label{eq:choice of r0}
0<r_f\le \frac{\min\{R,s_1\}}{2}
\end{equation}
such that
\begin{equation}\label{eq:small portion of E}
\frac{\mu(E\cap B(x,2r_f))}{\mu(B(x,2r_f))}\le \frac{1}{2C_d^{\lceil \log_2(60\lambda) \rceil}}
\end{equation}
and
\begin{equation}\label{eq:large IE minus IF}
r_f\frac{\rcapa_1((I_E\setminus\overline{I_F}^1)\cap B(x,r_f)
	,B(x,2r_f))}{\mu(B(x,r_f))}>a
\end{equation}
and
\begin{equation}\label{eq:small partial IF}
r_f\frac{\mathcal H(\partial^1 I_F\cap B(x,2r_f))}{\mu(B(x,r_f))}< \frac{a}{16 C_d^{\lceil\log_2(10\lambda)\rceil+2}C_P}.
\end{equation}
\subparagraph{Step 2.}
Let $D$ consist of all points 
$z\in B(x,r_f)\setminus \overline{I_F}^1$
for which there exists a radius $0<t\le  (10 \lambda)^{-1}r_f $ such that
\[
\frac{\mu(F\cap B(z,t))}{\mu(B(z,t))}> \frac{1}{4 C_d}.
\]
Consider $z\in D$ and the corresponding radius $t$. Since $z\notin \overline{I_F}^1\supset I_F\cup\partial^*F$
(recall \eqref{eq:meas theor bdry part of fine boundary}),
we have
\[
\lim_{r\to 0}\frac{\mu(F\cap B(z,r))}{\mu(B(z,r))}=0.
\]
Take the smallest $k=0,1,\ldots$ such that
\begin{equation}\label{eq:choice of k}
\frac{\mu(F\cap B(z,2^{-k}t))}{\mu(B(z,2^{-k}t))}\le\frac{1}{2}.
\end{equation}
If $k=0$, let $r_z:=t$ so that
\[
\frac{1}{4C_d}<\frac{\mu(F\cap B(z,r_z))}{\mu(B(z,r_z))}\le\frac{1}{2}.
\]
If $k\ge 1$, let $r_z:=2^{-k+1}t$, and then
\[
\frac{\mu(F\cap B(z,r_z))}{\mu(B(z,r_z))}>\frac{1}{2}
\]
and
\begin{align*}
\frac{\mu(F\cap B(z,r_z))}{\mu(B(z,r_z))}
&=\frac{\mu(F\cap B(z,2^{-k+1}t))}{\mu(B(z,2^{-k+1}t))}\\
&\le \frac{\mu(B(z,2^{-k+1}t))-\mu(B(z,2^{-k}t)\setminus F)}{\mu(B(z,2^{-k+1}t))}\\
&\le \frac{\mu(B(z,2^{-k+1}t))-\mu(B(z,2^{-k}t))/2}{\mu(B(z,2^{-k+1}t))}\quad\textrm{by }\eqref{eq:choice of k}\\
&\le 1-\frac{1}{2C_d}.
\end{align*}
Thus in both cases, we have $r_z\le (10 \lambda)^{-1}r_f$ and
\[
\frac{1}{4C_d}<\frac{\mu(F\cap B(z,r_z))}{\mu(B(z,r_z))}\le 1-\frac{1}{2C_d}.
\]
By the relative isoperimetric inequality
\eqref{eq:relative isoperimetric inequality for fine boundary} we now obtain
\begin{equation}\label{eq:condition on balls z rz and IF}
\mu(B(z,r_z)) \le 8 C_d^2 C_P r_z
\mathcal H(\partial^1 I_F\cap B(x,\lambda r_z)).
\end{equation}
Performing the same for every $z\in D$, we obtain a covering $\{B(z,\lambda r_z)\}_{z\in D}$.
By the $5$-covering theorem, we can extract a countable collection
$\{B_j=B(z_j,r_j)\}_{j=1}^{\infty}$ such that
the balls $\lambda B_j$ are pairwise disjoint and
$D\subset \bigcup_{j=1}^{\infty}5\lambda B_j$.
For each $j\in\N$, define the Lipschitz function
\[
\eta_j:=\max\left\{0,1-\frac{\dist(\cdot,5\lambda B_j)}{5\lambda r_j}\right\},
\]
so that
$\eta_j=1$ on $5\lambda B_j$, $\eta_j=0$ outside $10\lambda B_j$,
and the minimal $1$-weak upper gradient satisfies
$g_{\eta_j}\le (5\lambda r_j)^{-1}\ch_{10\lambda B_j}$ (see \cite[Corollary 2.21]{BB}).
Moreover, $r_j\le (10\lambda)^{-1}r_f$ and so
$\eta_j\in N_0^{1,1}(B(x,2r_f))$ for all $j\in\N$.
Now we have
\begin{align*}
\rcapa_1(D,B(x,2r_f))
&\le \rcapa_1\left(\bigcup_{j=1}^{\infty}5\lambda B_j,B(x,2r_f)\right)\\
&\le \sum_{j=1}^{\infty}\rcapa_1\left(5\lambda B_j,B(x,2r_f)\right)\\
&\le \sum_{j=1}^{\infty}\int_X g_{\eta_j}\,d\mu\\
&\le \sum_{j=1}^{\infty}\frac{\mu(10\lambda B_j)}{5\lambda r_j}\\
&\le C_d^{\lceil\log_2(10\lambda)\rceil}\sum_{j=1}^{\infty}\frac{\mu(B_j)}{r_j}\\
&\le 8C_d^{\lceil\log_2(10\lambda)\rceil+2}C_P\sum_{j=1}^{\infty}\mathcal H(\partial^1 I_F\cap \lambda B_j)\quad\textrm{by }
\eqref{eq:condition on balls z rz and IF}\\
&\le 8C_d^{\lceil\log_2(10\lambda)\rceil+2}C_P\mathcal H(\partial^1 I_F\cap B(x,2r_f)).
\end{align*}
Thus by \eqref{eq:small partial IF},
\[
r_f\frac{\rcapa_1(D,B(x,2r_f))}{\mu(B(x,r_f))}<\frac{a}{2}
\]
and so
\begin{equation}\label{eq:cap estimate for IE IF and D}
\begin{split}
& r_f\frac{\rcapa_1((I_E\setminus
	(\overline{I_F}^1\cup D))\cap B(x,r_f),B(x,2r_f))}{\mu(B(x,r_f))}\\
	&\qquad\ge r_f\frac{\rcapa_1((I_E\setminus
	\overline{I_F}^1)\cap B(x,r_f),B(x,2r_f))-\rcapa_1(D,B(x,2r_f))}{\mu(B(x,r_f))}\\
	&\qquad>\frac{a}{2}
\end{split}
\end{equation}
by \eqref{eq:large IE minus IF}.

\subparagraph{Step 3.}
Now consider $z\in (I_E\setminus (\overline{I_F}^1\cup D))\cap B(x,r_f)$. We have
\[
\lim_{r\to 0}\frac{\mu(E\cap B(z,r))}{\mu(B(z,r))}=1,
\]
and so we can choose $0<t\le (20\lambda)^{-1}r_f$ such that
\[
\frac{\mu(E\cap B(z,t))}{\mu(B(z,t))}>\frac{1}{2}.
\]
Note also that for any $r\in [(20\lambda)^{-1}r_f,(10\lambda)^{-1}r_f]$,
we have $B(x,2r_f)\subset B(z,60\lambda r)$ and so
\[
\frac{\mu(E\cap B(z,r))}{\mu(B(z,r))}\le 
C_d^{\lceil \log_2(60\lambda) \rceil} \frac{\mu(E\cap B(x,2r_f))}{\mu(B(x,2r_f))}\le \frac{1}{2}
\]
by \eqref{eq:small portion of E}.
Set $r_z:=2^{k}t$ for the smallest $k\in\N$ such that
\[
\frac{\mu(E\cap B(z,2^{k}t))}{\mu(B(z,2^{k}t))}\le \frac{1}{2}.
\]
Then we have $0<r_z\le (10\lambda)^{-1}r_f$ and
\begin{equation}\label{eq:proportion of E in ball}
\frac{1}{2C_d}< \frac{\mu(E\cap B(z,r_z))}{\mu(B(z,r_z))}\le \frac{1}{2}
\end{equation}
and since $z\notin D$,
\[
\frac{1}{4C_d}\le \frac{\mu((E\setminus F)\cap B(z,r_z))}{\mu(B(z,r_z))}\le \frac{1}{2}.
\]
Then by the relative isoperimetric inequality
\eqref{eq:relative isoperimetric inequality for fine boundary}, we have
\begin{equation}\label{eq:rel isop ineq for IE minus IF}
\frac{\mu(B(z,r_z))}{r_z}\le 8C_d^2 C_P\mathcal H(\partial^1 I_{E\setminus F}\cap B(z,\lambda r_z)).
\end{equation}
(Note that the right-hand side could be infinity.)
Let
\begin{equation}\label{eq:definition of A}
A:=\bigcup_{z\in (I_E\setminus (\overline{I_F}^1\cup D))\cap B(x,r_f)}
B(z,\lambda r_z)\subset B(x,2r_f).
\end{equation}
Consider the covering $\{B(z,\lambda r_z)\}_{z\in (I_E\setminus (\overline{I_F}^1\cup D))\cap B(x,r_f)}$.
By the $5$-covering theorem, we can extract a countable collection
$\{B_j=B(z_j,r_j)\}_{j=1}^{\infty}$ such that the balls
$\lambda B_j$ are pairwise disjoint and
$(I_E\setminus (\overline{I_F}^1\cup D))\cap B(x,r_f)
\subset \bigcup_{j=1}^{\infty}5\lambda B_j$.
Just as in the previous step,
for each $j\in\N$ define the Lipschitz function
\[
\eta_j:=\max\left\{0,1-\frac{\dist(\cdot,5\lambda B_j)}{5\lambda r_j}\right\},
\]
so that
$\eta_j=1$ on $5\lambda B_j$, $\eta_j=0$ outside $10\lambda B_j$,
and the minimal $1$-weak upper gradient satisfies $g_{\eta_j}\le (5\lambda r_j)^{-1}\ch_{10\lambda B_j}$.
Moreover, $r_j\le (10\lambda)^{-1}r_f$ and so $\eta_j\in N_0^{1,1}(B(x,2r_f))$ for all $j\in\N$. Now we have
\begin{equation}\label{eq:cap estimate for IE IF D 2}
\begin{split}
&\rcapa_1((I_E\setminus
	(\overline{I_F}^1\cup D))\cap B(x,r_f),B(x,2r_f))\\
&\qquad\qquad\le \rcapa_1\left(\bigcup_{j=1}^{\infty}5\lambda B_j,B(x,2r_f)\right)\\
&\qquad\qquad\le \sum_{j=1}^{\infty}\rcapa_1\left(5\lambda B_j,B(x,2r_f)\right)\\
&\qquad\qquad\le \sum_{j=1}^{\infty}\int_X g_{\eta_j}\,d\mu\\
&\qquad\qquad\le \sum_{j=1}^{\infty}\frac{\mu(10\lambda B_j)}{5\lambda r_j}\\
&\qquad\qquad\le C_d^{\lceil\log_2(10\lambda)\rceil}\sum_{j=1}^{\infty}\frac{\mu(B_j)}{r_j}\\
&\qquad\qquad\le 8C_d^{\lceil\log_2(10\lambda)\rceil+2}C_P\sum_{j=1}^{\infty}\mathcal H(\partial^1 I_{E\setminus F}\cap \lambda B_j)\quad\textrm{by }\eqref{eq:rel isop ineq for IE minus IF}\\
&\qquad\qquad\le 8C_d^{\lceil\log_2(10\lambda)\rceil+2}C_P\mathcal H(\partial^1 I_{E\setminus F}\cap A).
\end{split}
\end{equation}

\subparagraph{Step 4.}
Next we show that
\begin{equation}\label{eq:boundary of IE IF}
\partial^1 I_{E\setminus F}\subset
(\partial^1 I_E\setminus \overline{I_F}^1)\cup \partial^1 I_F.
\end{equation}
To see this, note that $X\setminus \overline{I_F}^1\subset O_F$
(recall \eqref{eq:meas theor bdry part of fine boundary})
and so $I_{E\setminus F}\setminus \overline{I_F}^1=I_{E}\setminus \overline{I_F}^1$.
Since $X\setminus \overline{I_F}^1$ is a $1$-finely open set,
it follows that $\partial^1 I_{E\setminus F}\setminus \overline{I_F}^1
=\partial^1 I_E\setminus \overline{I_F}^1$.
Moreover, $I_{E\setminus F}\cap \fint I_F=\emptyset$
and $\fint I_F$ is $1$-finely open, and so
$\partial^1 I_{E\setminus F}\cap \fint I_F
=\emptyset$.
From these, \eqref{eq:boundary of IE IF} follows.

By \eqref{eq:cap estimate for IE IF and D} and
\eqref{eq:cap estimate for IE IF D 2},
\[
r_f\frac{\mathcal H(\partial^1 I_{E\setminus F}\cap A)}{\mu(B(x,r_f))}\ge 
\frac{a}{16C_d^{\lceil\log_2(10\lambda)\rceil+2}C_P}.
\]
Now by first using \eqref{eq:boundary of IE IF} and the fact that
$A\subset B(x,2r_f)$ (recall \eqref{eq:definition of A}),
and then \eqref{eq:small partial IF} and the above inequality, we get
\[
r_f\frac{\mathcal H((\partial^1 I_E\setminus \overline{I_F}^1)\cap A)}{\mu(B(x,r_f))}
\ge r_f\frac{\mathcal H(\partial^1 I_{E\setminus F}\cap A)-
\mathcal H(\partial^1 I_F\cap B(x,2r_f))}{\mu(B(x,r_f))}>0.
\]
In particular, there exists a point $y\in (\partial^1 I_E\setminus \overline{I_F}^1)\cap A$.
Recall from  \eqref{eq:choice of r0} that
$0<r_f\le \min\{R,s_1\}/2$, and so $\overline{W}^1\cap B(x,2r_f)\subset I_F\cap B(x,2r_f)$.
Thus
\[
y\in (\partial^1 I_E\setminus \overline{W}^1)\cap B(x,2r_f)\subset 
(\partial^1 I_E\setminus \overline{W}^1)\cap B(x,s_1),
\]
as desired.
By the definition of $A$ (recall \eqref{eq:proportion of E in ball},
\eqref{eq:definition of A})
there is a point $z\in X$ and a radius $0<r_z\le (10\lambda)^{-1}r_f$ such that
$y\in B(z,\lambda r_z)$ and
\[
\frac{1}{2C_d}\le \frac{\mu(E\cap B(z,r_z))}{\mu(B(z,r_z))}\le \frac{1}{2}.
\]
For $s_2:=2\lambda r_z\le s_1/2$ we then have
$B(z,r_z)\subset B(y,s_2)\subset B(z,3\lambda r_z)$, and so
\[
\frac{1}{2C_d^{\lceil\log_2(3\lambda)\rceil+1}}
\le \frac{\mu(E\cap B(y,s_2))}{\mu(B(y,s_2))}
\le \frac{\mu(B(y,s_2))-\mu(B(z,r_z))/2}{\mu(B(y,s_2))}
\le 1-\frac{1}{2C_d^{\lceil\log_2(3\lambda)\rceil}}.
\]
This completes the proof of the claim.\\

Define $r_0=1$. We use the claim repeatedly, first with the choice $x=x_0$
and $s_1=r_0$, to find a sequence of points
$x_j\in  B(x_{j-1},r_{j-1})\cap \partial^1 
I_E\setminus \overline{W}^1$
and a sequence of numbers $0<r_j\le r_{j-1}/2$ such that
\[
\min\left\{\frac{\mu(B(x_j,r_j)\cap E)}{\mu(B(x_j,r_j))},
\frac{\mu(B(x_j,r_j)\setminus E)}{\mu(B(x_j,r_j))}\right\}
\ge \frac{1}{2C_d^{\lceil\log_2(3\lambda)\rceil+1}}
\]
for all $j\in\N$.
By completeness of the space and the fact that $W$ is open,
we find $x\in X\setminus W$ such that
$x_j\to x$. For each $j\in\N$ we have
\[
d(x,x_j)\le \sum_{k=j}^{\infty}d(x_k,x_{k+1})
\le \sum_{k=j}^{\infty}r_k\le 2r_j.
\]
Thus $B(x_j,r_j)\subset B(x,3r_j)\subset B(x_j,5r_j)$ for all $j\in\N$, and so
\[
\frac{\mu(B(x,3r_j)\cap E)}{\mu(B(x,3r_j))}
\ge \frac{\mu(B(x_j,r_j)\cap E)}{\mu(B(x,3r_j))}
\ge \frac{1}{C_d^3}\frac{\mu(B(x_j,r_j)\cap E)}{\mu(B(x_j,r_j))}
\ge \frac{1}{2C_d^{\lceil\log_2(3\lambda)\rceil+4}}
\]
and similarly
\[
\frac{\mu(B(x,3r_j)\setminus E)}{\mu(B(x,3r_j))}
\ge \frac{\mu(B(x_j,r_j)\setminus E)}{\mu(B(x,3r_j))}
\ge \frac{1}{C_d^3}\frac{\mu(B(x_j,r_j)\setminus E)}{\mu(B(x_j,r_j))}
\ge \frac{1}{2C_d^{\lceil\log_2(3\lambda)\rceil+4}}.
\]
Thus
\[
\limsup_{r\to 0}\frac{\mu(B(x,r)\cap E)}{\mu(B(x,r))}>0\quad\textrm{and}\quad
\limsup_{r\to 0}\frac{\mu(B(x,r)\setminus E)}{\mu(B(x,r))}>0,
\]
and so $x\in \partial^*E\setminus W$, which proves the theorem by the discussion
in the first paragraph of the proof.
\end{proof}

By using another argument involving Lipschitz cutoff functions, it is easy to see that for any $A\subset X$ and any ball $B(x,r)$,
\begin{equation}\label{eq:capacity and Hausdorff measure}
\rcapa_1(A\cap B(x,r),B(x,2r))\le C_d \mathcal H(A\cap B(x,r)).
\end{equation}

\begin{theorem}\label{thm:fine boundary when meas th boundary is finite}
	Let $\Om\subset X$ be open and let $E\subset X$ be $\mu$-measurable with
	$\mathcal H(\partial^*E\cap \Om)<\infty$.
	Then $\mathcal H((\partial^1 I_E\setminus \partial^*E)\cap \Om)=0$.
\end{theorem}
\begin{proof}
By a standard covering argument (see e.g. the proof of \cite[Lemma 2.6]{KKST-P})
we find that
\[
\lim_{r\to 0}r\frac{\mathcal H(\partial^*E\cap B(x,r))}{\mu(B(x,r))}=0
\]
for all $x\in \Om\setminus (\partial^*E\cup N)$, with $\mathcal H(N)=0$. Then by \eqref{eq:capacity and Hausdorff measure}, also
\begin{align*}
\limsup_{r\to 0}r\frac{\rcapa_1((\partial^*E\cup N)\cap B(x,r),B(x,2r))}{\mu(B(x,r))}
&\le C_d\limsup_{r\to 0}r\frac{\mathcal H((\partial^*E\cup N)\cap B(x,r))}{\mu(B(x,r))}\\
& = C_d\limsup_{r\to 0}r\frac{\mathcal H(\partial^*E\cap B(x,r))}{\mu(B(x,r))}\\
&=0
\end{align*}
for all $x\in \Om\setminus (\partial^*E\cup N)$.
Thus $\Om\setminus (\partial^*E\cup N)$ is a $1$-finely open set.
Now by Theorem \ref{thm:fine closure gives fine boundary},
$\partial^1 I_E \cap (\Om\setminus (\partial^*E\cup N))=\emptyset$ and the result follows.	
\end{proof}

Now we can prove our main theorem.

\begin{proof}[Proof of Theorem \ref{thm:Federers characterization}]
By Theorem \ref{thm:fine boundary when meas th boundary is finite} we have
$\mathcal H(\partial^1 I_E\cap \Om)<\infty$. Then by Theorem \ref{thm:Fed style characterization} we
have $P(E,\Om)<\infty$.
\end{proof}

\section{Some consequences and discussion}\label{sec:consequences}

Now we can state Federer's characterization in metric spaces
as follows.

\begin{corollary}
Let $\Om\subset X$ be open and let $E\subset X$ be $\mu$-measurable.
Then $P(E,\Om)<\infty$ if and only if $\mathcal H(\partial^*E\cap \Om)<\infty$.
\end{corollary}
\begin{proof}
This follows by combining Theorem \ref{thm:Federers characterization} and
\eqref{eq:def of theta}.
\end{proof}

In general, the sets $\partial^*E$ and $\partial^1 I_E$ can be quite different.

\begin{example}
Let $X=\R$ (unweighted).
Let $\{q_j\}_{j=1}^{\infty}$ be an enumeration of all rational
numbers and
let $E:=\bigcup_{j=1}^{\infty}B(q_j,2^{-j})$. Then
$\mathcal L^1(I_E)\le 2$ and
$\mathcal L^1(\partial^*E)=0$ by Lebesgue's differentiation theorem.
On the other hand, it is straightforward to check that for any
$A\subset \R$,
we have $\partial^1 A=\partial A$. Thus
$\partial^1 I_E=\partial I_E \supset \R\setminus I_E$ and so
$\mathcal L^1(\partial^1 I_E)=\infty$.
\end{example}

In the Euclidean setting, the ``if'' direction of
Federer's characterization is proved
by first showing
that almost every coordinate line intersecting $I_E$ and $O_E$ also intersects
$\partial^*E$, see \cite[Section 4.5.11]{Fed} or \cite[p. 222--]{EvGa}.
Proving this fact relies heavily on the Euclidean structure, and so it is difficult to
generalize to metric spaces. However, we do have the following;
it would be interesting to know if the assumption
$\mathcal H(\partial^*E)<\infty$ can be dropped.

\begin{proposition}
Let $E\subset X$ be $\mu$-measurable and
suppose that $\mathcal H(\partial^*E)<\infty$. Then $1$-almost every curve
 intersecting $I_E$ and $O_E$ also intersects
$\partial^*E$.
\end{proposition}
\begin{proof}
By Theorem \ref{thm:Federers characterization}, $P(E,X)<\infty$. Then the result
follows from \cite[Corollary 6.4]{LaSh}.
\end{proof}

It is reasonable to expect Federer's characterization
to find various applications especially in the metric setting,
where certain tools of Euclidean $\BV$ theory, such as the
Gauss-Green theorem, are not available.
One likely application is in the study of
images of sets of finite perimeter under
quasiconformal mappings (see \cite{Kel} for the Euclidean case),
since such mappings are known to preserve the
measure-theoretic boundary (see \cite[Theorem 6.1]{KoMaSh}).

Now we discuss some existing applications.
From the characterization it follows that
the space supports the following \emph{strong relative
isoperimetric inequality} introduced in \cite{KKST};
compare this with \eqref{eq:relative isoperimetric inequality}
and \eqref{eq:relative isoperimetric inequality for fine boundary}.

\begin{corollary}
For every ball $B(x,r)$ and every $\mu$-measurable $E\subset X$, we have
\[
\min\{\mu(B(x,r)\cap E),\,\mu(B(x,r)\setminus E)\}\le 2 C_P C_d r
\mathcal H(\partial^*E\cap B(x,\lambda r)).
\]
\end{corollary}
\begin{proof}
We can assume that the right-hand side is finite. By Theorem \ref{thm:Federers characterization}
we know that $P(E,B(x,\lambda r))<\infty$, and now the result follows by combining the relative
isoperimetric inequality \eqref{eq:relative isoperimetric inequality} and
\eqref{eq:def of theta}.
\end{proof}

In \cite{KKST} the authors worked with the same standing assumptions
as we do in the current paper,
but additionally they assumed that the space supports the above
strong relative isoperimetric inequality.
Now we know that this does not need to be separately assumed,
and the following theorem
(Theorem 1.1 of \cite{KKST}) holds under our standing assumptions
(completeness, doubling, and Poincar\'e).

\begin{theorem}
Let $\Om\subset X$ be a bounded open set and let $u\in N^{1,p}(\Om)$ with
$1\le p<\infty$. Then $u\in N_0^{1,p}(\Om)$ if and only if
\[
\lim_{r\to 0}\frac{1}{\mu(B(x,r))}\int_{\Om\cap B(x,r)}|u|\,d\mu=0
\]
for $\capa_p$-almost every $x\in \partial\Om$.
\end{theorem}

Theorem 6.1 in  \cite{LaSh2} considered an analogous
characterization of a class of $\BV$
functions with zero boundary values, also under the additional
assumption of a strong relative isoperimetric inequality.
Such a class is needed in an ongoing study of new fine properties of BV functions and capacities (begun in \cite{L-ZB}), and this was in fact a key
motivation for the current paper.
The strong relative isoperimetric inequality was also
used in proving approximation results for $\BV$ functions, see
\cite[Corollary 6.7, Theorem 6.9]{LaSh2} as well as \cite[Corollary 7.6]{LaSh} and the comment after it.
Now we know that all of these results hold in every complete
metric space equipped with a doubling measure and
supporting a Poincar\'e inequality.

\noindent Address:\\

\noindent University of Jyvaskyla\\
Department of Mathematics and Statistics\\
P.O. Box 35, FI-40014 University of Jyvaskyla\\
E-mail: {\tt panu.k.lahti@jyu.fi}


\begin{thebibliography}{ACMM}

\bibitem{AH}D. Adams and L. I. Hedberg,
\textit{Function spaces and potential theory},
Grundlehren der Mathematischen Wissenschaften, 314. Springer-Verlag, Berlin,
1996. xii+366 pp.

\bibitem{A1}L. Ambrosio,
\textit{Fine properties of sets of finite perimeter in doubling metric measure spaces},
Calculus of variations, nonsmooth analysis and related topics.
Set-Valued Anal. 10 (2002), no. 2-3, 111--128.

\bibitem{AFP}L. Ambrosio, N. Fusco, and D. Pallara,
\textit{Functions of bounded variation and free discontinuity problems.}
Oxford Mathematical Monographs. The Clarendon Press, Oxford University Press, New York, 2000.

\bibitem{AMP}L. Ambrosio, M. Miranda, Jr., and D. Pallara,
\textit{Special functions of bounded variation in doubling metric measure spaces},
Calculus of variations: topics from the mathematical heritage of E. De Giorgi, 1--45,
Quad. Mat., 14, Dept. Math., Seconda Univ. Napoli, Caserta, 2004.

\bibitem{BB}A. Bj\"orn and J. Bj\"orn,
\textit{Nonlinear potential theory on metric spaces},
EMS Tracts in Mathematics, 17. European Mathematical Society (EMS), Z\"urich, 2011. xii+403 pp.

\bibitem{BB-OD}A. Bj\"orn and J. Bj\"orn,
\textit{Obstacle and Dirichlet problems on arbitrary nonopen sets in metric spaces, and fine topology},
Rev. Mat. Iberoam. 31 (2015), no. 1, 161--214.

\bibitem{BBL-SS}A. Bj\"orn, J. Bj\"orn, and V. Latvala,
\textit{Sobolev spaces, fine gradients and quasicontinuity on quasiopen sets},
Ann. Acad. Sci. Fenn. Math. 41 (2016), no. 2, 551--560.

\bibitem{BBL-CCK}A. Bj\"orn, J. Bj\"orn, and V. Latvala,
\textit{The Cartan, Choquet and Kellogg properties for the fine topology on metric spaces},
to appear in J. Anal. Math.

\bibitem{BBL-WC}A. Bj\"orn, J. Bj\"orn, and V. Latvala,
\textit{The weak Cartan property for the p-fine topology on metric spaces},
Indiana Univ. Math. J. 64 (2015), no. 3, 915--941.

\bibitem{EvGa}L. C. Evans and R. F. Gariepy,
\textit{Measure theory and fine properties of functions},
Studies in Advanced Mathematics series, CRC Press, Boca Raton, 1992.

\bibitem{Fed}H. Federer,
\textit{Geometric measure theory},
Die Grundlehren der mathematischen Wissenschaften, Band 153 Springer-Verlag New York Inc., New York 1969 xiv+676 pp. 

\bibitem{Giu84}E. Giusti,
\textit{Minimal surfaces and functions of bounded variation},
Monographs in Mathematics, 80. Birkh\"auser Verlag, Basel, 1984. xii+240 pp.

\bibitem{HKM}J. Heinonen, T. Kilpel\"ainen, and O. Martio,
\textit{Nonlinear potential theory of degenerate elliptic equations},
Unabridged republication of the 1993 original. Dover Publications, Inc., Mineola, NY, 2006. xii+404 pp.

\bibitem{HK}J. Heinonen and P. Koskela,
\textit{Quasiconformal maps in metric spaces with controlled geometry},
Acta Math. 181 (1998), no. 1, 1--61.

\bibitem{Kel}J. Kelly,
\textit{Quasiconformal mappings and sets of finite perimeter},
Trans. Amer. Math. Soc. 180 (1973), 367--387. 

\bibitem{KKST}J. Kinnunen, R. Korte, N. Shanmugalingam, and H. Tuominen,
\textit{A characterization of Newtonian functions with zero boundary values},
Calc. Var. Partial Differential Equations 43 (2012), no. 3-4, 507--528.

\bibitem{KKST-P}J. Kinnunen, R. Korte, N. Shanmugalingam, and H. Tuominen,
\textit{Pointwise properties of functions of bounded variation in metric spaces},
Rev. Mat. Complut. 27 (2014), no. 1, 41--67.

\bibitem{KoLa}R. Korte and P. Lahti,
\textit{Relative isoperimetric inequalities and sufficient conditions for finite perimeter on metric spaces},
Ann. Inst. H. Poincar\'e Anal. Non Lin\'eaire 31 (2014), no. 1, 129--154.

\bibitem{KoLaSh}R. Korte, P. Lahti, and N. Shanmugalingam,
\textit{Semmes family of curves and a characterization of functions of bounded variation in terms of curves},
Calc. Var. Partial Differential Equations 54 (2015), no. 2, 1393--1424. 

\bibitem{KoMaSh}R. Korte, N. Marola, and N. Shanmugalingam,
\textit{Quasiconformality, homeomorphisms between metric measure spaces preserving quasiminimizers, and uniform density property},
Ark. Mat. 50 (2012), no. 1, 111--134. 

\bibitem{L-Fed}P. Lahti,
\textit{A Federer-style characterization of sets of finite perimeter on metric spaces},
Calc. Var. Partial Differential Equations 56 (2017), no. 5, Art. 150, 22 pp.

\bibitem{L-FC}P. Lahti,
\textit{A notion of fine continuity for BV functions on metric spaces},
Potential Anal. 46 (2017), no. 2, 279--294.

\bibitem{L-WC}P. Lahti,
\textit{Superminimizers and a weak Cartan property for $p=1$ in metric spaces},
to appear in J. Anal. Math.

\bibitem{L-ZB}P. Lahti,
\textit{The variational 1-capacity and BV functions with zero boundary values on metric spaces},
preprint 2017.
https://arxiv.org/abs/1708.09318

\bibitem{LaSh}P. Lahti and N. Shanmugalingam,
\textit{Fine properties and a notion of quasicontinuity for BV
	functions on metric spaces},
 J. Math. Pures Appl. (9) 107 (2017), no. 2, 150--182.

\bibitem{LaSh2}P. Lahti and N. Shanmugalingam,
\textit{Trace theorems for functions of bounded variation in metric spaces},
to appear in Journal of Functional Analysis.

\bibitem{LMZ}J. Luke\v{s}, J. Mal\'y, and L. Zaj\'i\v{c}ek,
\textit{Fine topology methods in real analysis and potential theory},
Lecture Notes in Mathematics, 1189. Springer-Verlag, Berlin, 1986. x+472 pp.

\bibitem{MZ}J. Mal\'{y} and W. Ziemer,
\textit{Fine regularity of solutions of elliptic partial differential equations},
Mathematical Surveys and Monographs, 51. American Mathematical Society, Providence, RI, 1997. xiv+291 pp.

\bibitem{M}M.~Miranda, Jr.,
\textit{Functions of bounded variation on ``good'' metric spaces},
J. Math. Pures Appl. (9) 82  (2003),  no. 8, 975--1004.

\bibitem{S}N. Shanmugalingam,
\textit{Newtonian spaces: An extension of Sobolev spaces to metric measure spaces},
Rev. Mat. Iberoamericana 16(2) (2000), 243--279.

\bibitem{Zie89}W. P. Ziemer,
\textit{Weakly differentiable functions. Sobolev spaces and functions of bounded variation},
Graduate Texts in Mathematics, 120. Springer-Verlag, New York, 1989. 




\end{thebibliography}
\end{document}